\documentclass[reqno]{amsart}
\usepackage{amsmath,amssymb,srcltx,enumerate,etoolbox,booktabs,mathrsfs}
\usepackage{paralist}
\usepackage[paperwidth=7in,paperheight=10in,tmargin=.9in,lmargin=.86in,rmargin=.86in,bmargin=17mm,asymmetric]{geometry}
\usepackage[flushmargin]{footmisc}
\usepackage{array}
\usepackage{tikz}
\usetikzlibrary{decorations.markings}

\newcommand{\dynkinradius}{.05cm}
\newcommand{\dynkinstep}{.4cm}
\newcommand{\dynkindot}[2]{\fill (\dynkinstep*#1,\dynkinstep*#2) circle (\dynkinradius);}
\newcommand{\dynkindotss}[2]{\draw[thick,dotted];}

\newcommand{\dynkinline}[4]{\draw[thin] (\dynkinstep*#1,\dynkinstep*#2) -- (\dynkinstep*#3,\dynkinstep*#4);}

\newenvironment{dynkin}{\begin{tikzpicture}[decoration={markings,mark=at position 0.7 with {\arrow{>}}}]}
{\end{tikzpicture}}
\patchcmd{\section}{\scshape}{\bfseries}{}{}
\patchcmd{\subsubsection}{\itshape}{\itshape}{}{}
\makeatletter
\renewcommand{\@secnumfont}{\bfseries}
\makeatother
\newtheorem{thm}{Theorem}
\newtheorem*{thm*}{Theorem}

\newtheorem{prop}[thm]{Proposition}
\newtheorem*{prop*}{Proposition}
\newtheorem{lem}[thm]{Lemma}
\newtheorem*{lem*}{Lemma}

\newtheorem*{cor*}{Corollary}

\newtheorem*{que*}{Question}
\theoremstyle{definition}
\newtheorem{Def}{Definition}
\newtheorem*{Def*}{Definition}
\newtheorem{remark}{Remark}
\newtheorem*{remark*}{Remark}

\newcommand{\ZZ}{\mathbf Z}
\newcommand{\CC}{\mathbf C}

\newcommand{\NN}{\mathbf N}
\newcommand{\GL}{{\rm GL}}

\newcommand{\type}{{\rm type}\,}

\newcommand{\rk}{n}

\begin{document}
\title[Walls in Milnor fiber complexes]{Walls in Milnor fiber complexes}
\author[A.~R.~Miller]{Alexander~R.~Miller}
\address{Fakult\"at f\"ur Mathematik, Universit\"at Wien, Wien, Austria}
\email{alexander.r.miller@univie.ac.at}
\begin{abstract}
For a real reflection group 
the reflecting hyperplanes 
cut out on the unit sphere a simplicial complex called 
the Coxeter complex. Abramenko showed that 
each reflecting hyperplane meets the Coxeter complex 
in another Coxeter complex if and only if the Coxeter 
diagram contains no subdiagram of type $D_4$, $F_4$, or $H_4$. 
The present paper extends Abramenko's result to a wider class of 
complex reflection groups. 
These groups have a Coxeter-like presentation and a Coxeter-like 
complex called the Milnor fiber complex. 
Our first main theorem classifies the groups 
whose reflecting hyperplanes meet the Milnor fiber complex 
in another Milnor fiber complex. 
To understand better the walls that fail to be Milnor fiber complexes 
we introduce {\it Milnor walls}. 
Our second main theorem generalizes Abramenko's result in 
a second way. It says 
that each wall of a Milnor fiber complex is a Milnor wall 
if and only if the diagram contains 
no subdiagram of type {$D_4$, $F_4$, or $H_4$.} 
\end{abstract}
\maketitle

\section{Introduction}
\noindent
For a real reflection group the reflecting hyperplanes 
cut out on the unit sphere a simplicial complex called 
the Coxeter complex. Abramenko \cite{A} showed that 
each reflecting hyperplane meets the Coxeter complex 
in another Coxeter complex if and only if the Coxeter 
diagram contains no subdiagram of type $D_4$, $F_4$, or $H_4$. 

The present paper extends Abramenko's result to a wider class of 
complex reflection groups. These groups have a Coxeter-like 
presentation and a Coxeter-like complex called the Milnor fiber complex. 
Our first main theorem (Theorem~\ref{Theorem:A}) classifies the groups 
whose reflecting hyperplanes meet the Milnor fiber complex 
in another Milnor fiber complex. 

To understand better the walls that fail to be Milnor fiber 
complexes we introduce {\it Milnor walls}.  These  
are walls with a type-selected set of chambers generating   
a Milnor fiber complex. Milnor walls 
in Coxeter complexes are walls that are Coxeter complexes. 
Our second main theorem (Theorem~\ref{Theorem:B}) thus 
generalizes Abramenko's result in a second way: it says 
each wall of a Milnor fiber complex is a Milnor wall 
if and only if the diagram contains no subdiagram of 
type $D_4$, $F_4$, or $H_4$. 

As a benefit of Theorem~\ref{Theorem:B} we find 
that Abramenko's result extends to give yet another equivalent 
condition in a curious classification \cite[Theorem 14]{M1} 
with already 11 equivalent conditions 
coming from invariant theory, cohomology, combinatorics, and some 
group characters related to adding random numbers. 
\subsection{}Fix a nonnegative integer $n\geq 0$ and a finite group $G$ of the form 
\begin{equation}
\langle\, r_1,r_2,\ldots,r_\rk \mid r_i^{p_i}=1,\ \  
r_ir_jr_i\ldots=r_jr_ir_j\ldots \ \ i\neq j\, \rangle\label{Presentation}
\end{equation}
where $p_i\geq 2$,   
the number of terms on both sides of the braid relation is 
${m_{ij}=m_{ji}\geq 2}$, and 
$p_i$ equals $p_j$ if $m_{ij}$ is odd\footnote{For $m_{ij}$ odd 
the braid relation says $(r_ir_j)^{(m_{ij}-1)/2}r_i=r_j(r_ir_j)^{(m_{ij}-1)/2}$
so that $r_i$ is conjugate to $r_j$.}. The empty set (when $n=0$) generates the trivial 
group.

\subsection{Preliminaries}
Call a group presentation of the form \eqref{Presentation} {\it admissible}. 
Koster~\cite{Ko} classified admissible presentations         
and found that the groups are precisely the finite direct products 
of finite irreducible Coxeter groups and Shephard groups. 
The classification implies no group has  two different admissible 
presentations. Write ${R = \{r_1 , r_2 , \ldots , r_\rk \}}$ and call 
$n$ the {\it rank} of~$G$.

\subsubsection{}The classification uses a    
graphical notation for admissible presentations. 
The {\it diagram} 
$\Gamma$ of \eqref{Presentation} 
has for each $r_i$ a vertex labeled $p_i$, and for 
each pair $r_i,r_j$ with $m_{ij}>2$ an edge labeled 
$m_{ij}$ that connects $r_i$ and $r_j$. We agree to 
suppress the minimal labels (2's~on vertices and 3's on edges). 
We say $\Gamma$ is {\it connected} if it has exactly one 
connected component; the diagram with no vertices is not connected. 
By {\it subdiagram of $\Gamma$} we mean a diagram 
gotten from $\Gamma$ by 
removing any number of vertices and all their incident edges.

\subsubsection{}Write $G=G(\Gamma)$ and let $\Gamma_1,\Gamma_2,\ldots,\Gamma_N$ 
be the connected components of $\Gamma$. Then 
\begin{equation}\label{irr:decomp}
G=G_1\times G_2\times\ldots\times G_N,\quad G_i=G(\Gamma_i)
\end{equation} 
where the empty product is the trivial group.
It follows that admissible diagrams are the unions of 
connected ones. Koster classified the connected ones. 
Table \ref{P:Table} lists them. The groups are the 
finite irreducible Coxeter groups and 
Shephard groups. Each comes from just one diagram.  
Finite irreducible Coxeter groups are the ones 
with all vertices~2.
Shephard groups are the ones with {\it linear diagram}
\begin{equation}\label{linear:diagram}
\begin{tikzpicture}
\draw [thick,-] (0,0)--(1.5,0);
\draw [thick,-] (1.5,0)--(3,0);
\draw [thick,-] (4+.25,0)--(5.5+.25,0);
\draw[fill] (0,0) circle [radius=0.04];
\draw[fill] (1.5,0) circle [radius=0.04];
\draw[fill] (3,0) circle [radius=0.04];
\draw[fill] (4+.25,0) circle [radius=0.04];
\draw[fill] (5.5+.25,0) circle [radius=0.04];
\node [above] at (.75,0) {\small$q_1$};
\node [above] at (2.25,0) {\small$q_2$};
\node [above] at (4.75+.15+.25,0) {\small$q_{n-1}$};
\node [above] at (0,0) {\small$p_1$};
\node [above] at (1.5,0) {\small$p_2$};
\node [above] at (3,0) {\small$p_3$};
\node [above] at (4+.15+.25,0) {\small$p_{n-1}$};
\node [above] at (5.5+.05+.25,0) {\small$p_{n}$};
\fill (3.5+.125,0) node {\ldots};
\end{tikzpicture}
\end{equation}
The {\it symbol} $p_1[q_1]p_2[q_2]p_3\ldots p_{n-1}[q_{n-1}]p_n$ 
(unique up to reversing term order) is shorthand for the linear 
diagram \eqref{linear:diagram}.

\subsubsection{}$G$ has a representation 
analogous to the canonical one for Coxeter groups. 
Fix a vector space $V$ over $\CC$ of dimension $n$. 
A {\it reflection} in $\GL(V)$ is 
an element of finite order whose fixed space $V^r=\ker(r-1)$ 
is a hyperplane, and a {\it finite reflection group} is a 
finite group generated by reflections. Finite Coxeter groups 
have a canonical representation as a real reflection group  
that we view as a reflection group by extending the base field.
In general the diagram $\Gamma$ encodes a canonical faithful 
representation of $G$ as a reflection group $G\subset \GL(V)$ 
in which each $r\in R$ is a reflection~\cite{Ko}. 
With this identification the reflections in $G$ are 
precisely the non-identity elements that are conjugate to a power  
of a generator $r\in R$. Call $G$ {\it irreducible} if 
the $\CC G$-module $V$ is irreducible. 
This happens if and only if $\Gamma$ is connected. 
Shephard groups are irreducible; the trivial group is not. 

\subsubsection{}Shephard and Todd classified the finite irreducible 
reflection groups and named {\it exceptional} ones 
$G_4,G_5,\ldots,G_{37}$. Not all of them are Coxeter or Shephard groups. 
The Coxeter ones have another set of names that we also use. For example 
$H_3$ and $G_{23}$ both refer to the same Shephard group in Table~\ref{P:Table}.

\subsubsection{}\label{Section:Degrees}
Finite reflection groups on $V$ 
are also the finite groups acting linearly on $V$ 
whose algebra of invariant polynomial functions $P$ on $V$ 
(with respect to $gP(v)=P(g^{-1}v)$)  
is generated by $n=\dim V$ homogeneous algebraically independent polynomials $P_i$. 
The {\it basic degrees} $d_i=\deg P_i$ are unique and numbered 
so that $d_1\leq d_2\leq\ldots\leq d_n$. 
If $G$ is irreducible, 
then by the classification (see Table~\ref{P:Table}):
\begin{enumerate}[(i)]
\item\label{Coxeter:Inequality} $d_1\geq 2$ with equality  if and only if $G$ is a Coxeter group.
\item\label{Milnor:Inequality} $d_1<d_2$ so that $P_1$ is unique up to a constant factor.
\end{enumerate}
If $G$ is irreducible, then $F_G=P_1^{-1}(1)$ is   
called {\it the Milnor fiber of~$G$}.

\subsubsection{}\label{Milnor:Intro}
According to details in \S\ref{MF:Details} 
there is a unique (up to $G$-isomorphism) 
abstract simplicial complex $\Delta=\Delta(G,R)$ with     
simplices (labeled by) the cosets $gG_I$ of standard parabolic 
subgroups $G_I=\langle I\rangle$ ($I\subset R$)  
with face relation 
``{$gG_I$ is a face of $hG_J$}'' $\Leftrightarrow$ $gG_I\supset hG_J$, and 
with $G$ acting on $\Delta$ by 
left translation $g.hG_I=ghG_I$. 
This is the classical abstract description of the 
Coxeter complex when $G$ is a Coxeter group \cite{Tits}.
Call $\Delta$ {\it the Milnor fiber complex of $G$}. 
It has an explicit geometric realization in $V$ that is $G$-homeomorphic 
to an equivariant strong deformation retract of the Milnor fiber $F_G$ if 
$G$ is irreducible \cite{O:Milnor}. In general it is described 
in \S\ref{MF:Details} as the join of the Milnor 
fiber complexes of the irreducible factors $G_i$ of~$G$.

Maximal simplices in $\Delta$ have dimension $n-1$ and any 
two can be connected by a sequence of them in which 
consecutive terms share a face of codimension $1$ so that $\Delta$ 
is a {\it chamber complex} and maximal simplices are called {\it chambers}. 
A general simplex $gG_{R\setminus I}$ has 
vertex set $\{gG_{R\setminus \{r\}}: r\in I\}$ and dimension $|I|-1$. 
The set $I\subset R$ is uniquely determined by $gG_{R\setminus I}$. 
Call $I$ the {\it type} of~$gG_{R\setminus I}$ and write 
$\type(gG_{R\setminus I})=I$. 

\subsubsection{}
The simplices of $\Delta$ that are fixed pointwise by a 
reflection of $G$ form a subcomplex we call a {\it wall}. 
Since an element $g\in G$ fixing a simplex $hG_{R\setminus I}\in \Delta$ 
effects a type-preserving permutation of the vertices 
$hG_{R\setminus\{r\}}$ ($r\in I$) the simplex is in fact 
fixed pointwise by $g$ and so the walls of $\Delta$ are 
\begin{equation}
\Delta^r=\{\sigma\in \Delta : r\sigma=\sigma\},\quad\text{$r\in G$ a reflection}.
\end{equation}

\subsection{} Our first theorem extends Abramenko's result to Milnor fiber complexes.

\begin{thm}\label{Theorem:A}
Each wall of the Milnor fiber complex $\Delta$ is again a Milnor fiber complex 
if and only if the diagram of~$G$ contains no subdiagram of type $D_4$, $F_4$, $H_4$, 
$G_{25}$, or~$G_{26}$. 
\end{thm}

\noindent
We prove Theorem~\ref{Theorem:A} in Section~\ref{S:Proof:A} by 
first reducing to the case where $G$ is irreducible and then 
using the classification together with some 
enumerative and topological results about $\Delta$ 
that relate to the invariant theory of $G$.

\subsubsection{}\label{Section:Invariants} 
We recover Abramenko's result from Theorem~\ref{Theorem:A} 
by Proposition~\ref{Lemma:Coxeter:Walls}, which tells us   
that for Coxeter complexes all walls that are Milnor fiber 
complexes must be Coxeter complexes.

\begin{cor*}[Abramenko]
Each wall of a Coxeter complex is again a Coxeter complex if and only if the 
diagram contains no subdiagram of type $D_4$, $F_4$, or $H_4$.\qed
\end{cor*}

\subsection{} Our second theorem generalizes Abramenko's result in 
another way. The observation is that 
walls in Milnor fiber complexes 
can still hold Milnor fiber complexes of the same 
dimension in the sense that 
certain types of chambers in the wall generate a Milnor fiber complex. 
We make this precise with the definition 
of {\it Milnor wall} in Section~\ref{S:Milnor:Walls}, and then 
we prove the following theorem which also implies Abramenko's result.
  
\begin{thm}\label{Theorem:B}
Each wall of the Milnor fiber complex $\Delta$ is a Milnor wall 
if and only if the diagram of~$G$ contains no subdiagram of type 
$D_4$, $F_4$, or $H_4$.
\end{thm}

\begin{remark} 
In \cite[Theorem 14]{M1} we proved that 
if $G$ is irreducible, then 
the diagram of $G$ contains no subdiagram 
of type $D_4$, $F_4$, or $H_4$ if and only if 
the {\it Foulkes characters} $\phi_0,\phi_1,\ldots,\phi_n$ for $G$ depend only on
fixed-space dimension in the sense that $\phi_i(g)=\phi_i(h)$ 
whenever $\dim V^g=\dim V^h$. 
See \cite[Theorem 14]{M1} for 9 other equivalent conditions.
\end{remark}

\section{Milnor fiber complexes}\label{S:Proof:A}
\noindent
In this section we prove Theorem~\ref{Theorem:A} after some preliminaries. 
\S\ref{MF:Details} defines the Milnor fiber complex.  
\S\ref{Section:Walls} defines walls.
\S\ref{SS:Top} develops some topological results. \S\ref{SS:Enum} develops 
some enumerative results. 
\S\ref{S:MF:Comb} gives a combinatorial description of the Milnor fiber complex 
for the full monomial groups. 
Then in \S\ref{Section:Proof:A} we prove Theorem~\ref{Theorem:A}.

Recall that 
the connected components $\Gamma_i$ of $\Gamma$ partition $R$ into 
disjoint sets $R_i$ so that 
\begin{equation}\label{G:Decomposition}
G=G_1\times G_2\times \ldots\times G_N,\quad G_i=\langle R_i\rangle.
\end{equation}
Let $n_i$ and $\delta_i$ be the rank and smallest basic degree of  $G_i$,  
so that ${n_i=|R_i|\geq 1}$,  ${n=n_1+n_2+\ldots+n_N}$, 
and $\delta_i\geq 2$ with equality if and only if $G_i$ is a Coxeter group.

\subsection{The Milnor fiber complex}\label{MF:Details} 
We define the {\it Milnor fiber complex of $G$} to be the 
complex $\Delta$ described by the following theorem.
The definition is in terms of cosets 
$gG_I$ of standard parabolic 
subgroups $G_I=\langle I\rangle$ ($I\subset R$). 
If $G$ is a Coxeter group, then the definition is 
the standard abstract one for the Coxeter complex of $G$ and 
the properties that we list are well known, see \cite{Tits}. 
The geometric construction of the Coxeter complex was generalized 
to include Shephard groups by Orlik \cite{O:Milnor}. 
He called this more general complex the Milnor fiber complex. 
The abstract definition of the Coxeter complex 
was later shown to hold for the Milnor fiber complex in the 
Shephard case~\cite{ORS}. 
For details about the extension of the definition and the properties to the 
Shephard case see~\cite{M0}. The general 
case of the following theorem follows from the Coxeter and Shephard cases.

\begin{thm}\label{Theorem:Definition} 
The standard parabolic cosets $gG_I$ ($g\in G, I\subset R$) 
with face relation 
\[\text{``$gG_I$ is a face of $hG_J$'' $\Leftrightarrow$ $gG_I\supset hG_J$}\] 
is a simplicial complex $\Delta$. Moreover $\Delta$ is a chamber 
complex with the following structure:
\begin{enumerate}[\rm(i)]
\item $G$ acts on $\Delta$ by left translation $g.hG_I=ghG_I$.
\item $\Delta$ has a $G$-invariant {\it type function} 
$\Delta\to\{\text{subsets of $R$}\}$ given by $\type(hG_{R\setminus I})=I$.
\item\label{Join:Isom} There exists a type-preserving $G$-equivariant isomorphism 
\begin{equation}\label{MFC:Def}
\Delta\cong\Delta_1*\Delta_2*\ldots*\Delta_N,\quad \Delta_k=\Delta(G_k,R_k)
\end{equation}
where 
$g.(h_1G_{I_1}*\ldots *h_NG_{I_N})=g_1h_1G_{I_1}*\ldots *g_Nh_NG_{I_N}$ for $g=g_1g_2\ldots g_N$,  $g_i\in G_i$, 
$I_i\subset R_i$, and where 
$\type(h_1G_{R_1\setminus I_1}*\ldots *h_NG_{R_N\setminus I_N})=I_1\cup\ldots\cup I_N$.
\end{enumerate}
\end{thm}

\begin{proof}
For the case where $G$ is a Coxeter or Shephard group see \cite{Tits} and \cite{M0}. 
For the general case it suffices to prove \eqref{Join:Isom}. 
To this end note that the mapping 
\[h_1G_{I_1}*h_2G_{I_2}*\ldots *h_NG_{I_N}\mapsto h_1h_2\ldots h_N G_{I_1\cup I_2\cup \ldots \cup I_N}\]
takes $\Delta_1*\Delta_2*\ldots*\Delta_N$ bijectively onto $\Delta$ in a type-preserving fashion, 
and is compatible with the face relation and $G$-action. In particular $\Delta$ is a simplicial complex. 
\end{proof}

We require the following lemma 
\cite[Lemma 3.14 with $T=R\setminus U$]{M0} which tells us 
that each link in a Milnor fiber complex is again a Milnor fiber complex.
\begin{prop}\label{Lemma:Products}
The link of a simplex $gG_I$ in $\Delta$ is isomorphic to ${\Delta(G_I,I)}$. \qed
\end{prop}

\subsection{The walls of the Milnor fiber complex}\label{Section:Walls}
The simplices of $\Delta$ that are fixed pointwise by a 
reflection of $G$ form a subcomplex we call a {\it wall}. 
The following proposition says that a simplex $\sigma\in \Delta$ 
is fixed pointwise by $g\in G$ if and only if $g\sigma=\sigma$. Write 
\begin{equation}
\Delta^g=\{\sigma\in \Delta : g\sigma=\sigma\},\quad g\in G.
\end{equation}

\begin{prop}\label{Prop:fixed:pointwise}
$\Delta^g=\{\sigma\in \Delta : \sigma\text{ \normalfont fixed pointwise by $g$}\}$.
\end{prop}

\begin{proof}
An element $g\in G$ fixing a simplex $hG_{R\setminus I}\in \Delta$ 
effects a type-preserving 
permutation of the vertices $hG_{R\setminus\{r\}}$ ($r\in I$) of the simplex. 
Since no two of these vertices have the same type, 
the element $g$ must fix each of the vertices.
\end{proof}

Proposition~\ref{Prop:fixed:pointwise} 
gives the following description of walls.

\begin{prop}\label{Prop:Walls:Desc}
The walls of $\Delta$ are 
\begin{equation}
\Delta^r=\{\sigma\in \Delta : r\sigma=\sigma\},\quad\text{$r\in G$ a reflection}.
\end{equation}
Equivalently, the walls of $\Delta$ are (up to isomorphism) 
\begin{equation}\label{Wall:Join}
\Delta(G_1,R_1) * \ldots *\Delta(G_i,R_i)^{t} * \ldots * \Delta(G_N,R_N),\quad 
t\in G_i\ \text{a reflection}.
\end{equation}
\end{prop}
\begin{proof}
The first part is by Proposition~\ref{Prop:fixed:pointwise}. 
The second part follows by Theorem~\ref{Theorem:Definition}\eqref{Join:Isom}.
\end{proof}

As a benefit of \eqref{Wall:Join} we have the 
following result that reduces the problem of 
determining when the walls of 
$\Delta$ are Milnor fiber complexes to the case when $G$ is irreducible. 

\begin{prop}\label{Proposition:Reduction:A}
Each wall of $\Delta(G,R)$ is a Milnor fiber complex if and only if 
each wall of each $\Delta(G_i,R_i)$ is a Milnor fiber complex. 
\end{prop}

\begin{proof}
If the wall in \eqref{Wall:Join} is a Milnor fiber complex, then 
it follows from Proposition~\ref{Lemma:Products} 
 that the wall 
$\Delta(G_i,R_i)^{t}$ of $\Delta(G_i,R_i)$ is a Milnor fiber complex. 
Conversely, if the wall $\Delta(G_i,R_i)^{t}$ of $\Delta(G_i,R_i)$ 
is a Milnor fiber complex 
$\Delta(G_i',R_i')$, then the wall in \eqref{Wall:Join} is 
the Milnor fiber complex of $G_1\times \ldots\times G_i'\times\ldots\times G_N$. 
\end{proof}

\subsection{Topological results}\label{SS:Top}
The following theorem tells us the homotopy type of the subcomplex $\Delta^g$ 
for any element $g\in G$. In particular, it tells us the homotopy type of 
any wall $\Delta^r$. It is due to Orlik~\cite{O:Milnor} and appears 
in the proof of his Theorem 4.1 on p.~145 where he observes that 
$\Delta^g$ is a deformation retract of the intersection of 
the fixed space $V^g$ and the Milnor fiber of $G$, 
which has an isolated critical point at the origin. 

\begin{thm*}[Orlik]
If $G$ is irreducible and $g\in G$, then the subcomplex $\Delta^g$ is 
homotopy equivalent to a bouquet of $(d_1-1)^p$ many $(p-1)$-spheres,  
where $p=\dim V^g$.
\end{thm*}

If $G$ is reducible, then the homotopy type of a given 
$\Delta^g$ is read off from Orlik's result and 
\eqref{MFC:Def}.  We highlight the case $g=1$. 
This case is used in the proof of Theorem~\ref{Theorem:A} to help 
determine if a wall $\Delta^r$ is 
a Milnor fiber complex $\Delta(W,S)$ for some possibly reducible~$W$. 

\begin{prop}\label{Cor:Orlik}
Let $n_i$ and $\delta_i$ be the rank and smallest basic degree of 
the irreducible factor $G_i$, so that $n=n_1+n_2+\ldots+n_N$. 
Then the Milnor fiber complex of $G$ is homotopy equivalent to a bouquet of 
$(\delta_1-1)^{n_1}(\delta_2-1)^{n_2}\ldots(\delta_N-1)^{n_N}$ many $(n-1)$-spheres. 
\end{prop}

\begin{proof}
The Milnor fiber complex is the join $\Delta_1*\Delta_2*\ldots*\Delta_N$ 
of the complexes $\Delta_i$ of the irreducible factors $G_i$, 
and  
Orlik's theorem with $g=1$ tells us that $\Delta_i$ is a bouquet of 
$(\delta_i-1)^{n_i}$ many $(n_i-1)$-spheres. Hence the result.
\end{proof}

From Proposition~\ref{Cor:Orlik} we get the 
following characterization of Coxeter complexes 
as Milnor fiber complexes that are spheres. 

\begin{prop}\label{Cor:Coxeter:Sphere}%
\! A Milnor fiber complex is a Coxeter complex if and only if it~is~a~sphere.%
\end{prop}

\begin{proof}
Proposition~\ref{Cor:Orlik} says $\Delta$ 
is a bouquet of $(\delta_1-1)^{n_1}(\delta_2-1)^{n_2}\ldots(\delta_N-1)^{n_N}$ 
many $(n-1)$-spheres. If $N=0$, then $\Delta$ is a single $(-1)$-sphere and $G$ 
is the trivial Coxeter group $\langle\emptyset\rangle$. If $N>0$, 
then the inequalities $n_i\geq 1$ and $\delta_i\geq 2$ 
imply that the number of spheres equals 1 if and only if each $\delta_i$ equals~$2$. 
Hence by \S\ref{Section:Degrees}\eqref{Coxeter:Inequality} the Milnor fiber complex 
$\Delta$ is a single sphere if and only if each $G_i$   
is a Coxeter group, i.e., if and only if $G$ is a Coxeter group. 
\end{proof}

As a corollary of Proposition~\ref{Cor:Coxeter:Sphere} 
we have the following.

\begin{prop}\label{Lemma:Coxeter:Walls}
A wall of a Coxeter complex is a Milnor fiber complex 
if and only if it is a Coxeter complex.
\end{prop}

\begin{proof} 
Consider a wall of a Coxeter complex. If it is a Coxeter complex, then 
it is a Milnor fiber complex. If it is a Milnor fiber complex, 
then because it is also a sphere (being the equator of a sphere) 
Proposition~\ref{Cor:Coxeter:Sphere} tells us that it is 
a Coxeter complex.
\end{proof}

\subsection{Enumerative results}\label{SS:Enum}
We require some enumerative results about the number of 
chambers of a wall. Denote by $f_k(\Sigma)$ the number of 
$k$-simplices in a complex $\Sigma$ so~that
\[f_k(\Sigma)=\#\{\sigma\in\Sigma : \dim \sigma=k\}.\]

The chambers of $\Delta$ are indexed by 
elements of $G$, and the number of 
elements of $G$ equals the product 
of basic degrees $d_1,d_2,\ldots,d_\rk$. 
Hence $f_{\rk-1}(\Delta)$ equals $d_1d_2\ldots d_\rk$. 
Suppose for the rest of \S\ref{SS:Enum} 
that $G$ is irreducible. 
Then Eq.~\eqref{Eq:Chamber:Count} below tells us that 
$f_{\rk-2}(\Delta^r)$ equals $d_1d_2\ldots d_{\rk-1}$ 
for any reflection $r$ in $G$. 
It is natural to wonder then 
if a similar formula holds for elements $g$ where 
$\dim V^g$ equals $\rk-2$, $\rk-3$, and so on. 
Remarkably this  
turns out to be the case if and only if the diagram of $G$ 
does not contain any subdiagram of type $D_4$, $F_4$, or $H_4$. This 
is Theorem~\ref{Theorem:Chamber:Count} below. It 
is a consequence of a collection of observations from~\cite{M1} 
about Orlik--Solomon coexponents. 

Continue to suppose that $G$ is irreducible. Let $L$ 
be the collection of all fixed spaces $V^g$  
ordered by reverse inclusion, so that $V$ is at the bottom. This is the same as the  
lattice of intersections of reflecting hyperplanes. 
For $\mu$ the M\"obius function of~$L$ and $X\in L$ 
define $B_X(t)\in\ZZ[t]$ by 
$B_X(t)=(-1)^{\dim X}\sum_{Y\geq X}\mu(X,Y)(-t)^{\dim Y}$. 
Then Orlik \cite{O:Milnor} (following Orlik--Solomon in the Coxeter case) showed that 
\begin{equation}
f_{k-1}(\Delta^g)=\sum B_Y(d_1-1)
\end{equation}
where the sum is over all $k$-dimensional subspaces $Y$ that lie above $V^g$ in $L$. 
In particular
\begin{equation}
f_{p-1}(\Delta^g)=B_X(d_1-1)\label{B:Chamber:Count}
\end{equation} 
for $X=V^g$ and $p=\dim X$. 
Furthermore for $X\in L$ of dimension $p$ there exist positive 
integers $b_1^X\leq b_2^X\leq\ldots \leq b_p^X$ such that 
\begin{equation}
B_X(t)=(t+b_1^X)(t+b_2^X)\ldots(t+b_p^X).
\end{equation}
Orlik and Solomon determined the $b_i^X$'s for all 
irreducible Coxeter and Shephard groups. 
The tables in \cite{OS:Unitary,OS:Cox} list the $b_i^X$'s for each $X\in L$ 
when $G$ is an exceptional group of rank at least 3. Following \cite{M1} 
we make the following remarkable observation.

\begin{thm}\label{Theorem:Chamber:Count}
If $G$ is an irreducible Coxeter or Shephard group, then 
\begin{equation}\label{Eq:Chamber:Count}
f_{\rk-2}(\Delta^r)=d_1d_2\ldots d_{\rk-1}
\end{equation}
for any reflection $r\in G$, and the following are equivalent:
\begin{enumerate}[\rm(i)]
\item $f_{\rk-3}(\Delta^g)=d_1d_2\ldots d_{\rk-2}$ for $g\in G$ such that $\dim V^g=\rk-2$.\label{Eq:Count:1}
\item $f_{p-1}(\Delta^g)=d_1d_2\ldots d_{p}$ for $g\in G$ and $p=\dim V^g$.\label{Eq:Count:2}
\item The diagram of $G$ contains no subdiagram of type $D_4$, $F_4$, or $H_4$.\label{Eq:Count:3}
\end{enumerate}  
\end{thm}

\begin{proof}
As stated above, this is a collection of observations from~\cite{M1} about 
Orlik--Solomon coexponents (see~\cite[Prop.~13(ii), Thm.~14(a)(k), Pf. of Thm.~14]{M1}) 
as we now explain. For irreducible Coxeter groups, Eq.~\eqref{Eq:Chamber:Count} 
was observed by Orlik and Solomon \cite[p.~271]{OS:Cox}. In general  
\cite[Thm.~14(g)(k)]{M1} is the equivalence  
\eqref{Eq:Count:2}$\Leftrightarrow$\eqref{Eq:Count:3}. 
So it remains to show that 
\eqref{Eq:Count:1} fails for $F_4$, $H_4$, $E_6$, $E_7$, $E_8$, and $D_\rk\ (\rk\geq 4)$. 
This is implicit in the proof of \cite[Thm.~14]{M1}. The 
exceptional cases $F_4$, $H_4$, $E_6$, $E_7$, and $E_8$ follow from \eqref{B:Chamber:Count} 
together with the tables of \cite{OS:Cox}, and 
as explained in the proof of \cite[Thm.~14]{M1},  
\eqref{Eq:Count:1} fails for $D_\rk$ since 
there is a certain fixed pace $Y$ defined by $x_1=x_2=x_3$ such that 
$B_Y(d_1-1)<d_1d_2\ldots d_{\rk-2}$. 
\end{proof}

\subsection{Another description of the Milnor fiber complex of the full monomial group}\label{S:MF:Comb}
Fix an integer ${m>1}$. The {\it full monomial group} $G(m,1,n)$ 
is the group of all $n$-by-$n$ monomial matrices (one nonzero 
entry in each row and column) whose nonzero entries are $m$-th 
roots of unity. Let $\zeta$ be a primitive $m$-th root of unity 
and denote by $e_k$ the standard column vector in $\CC^n$ with 
$1$ in the $i$-th spot and 0 elsewhere. In cycle notation, the 
standard generators $r_1,r_2,\ldots, r_n$ of $G(m,1,n)$ are the 
adjacent transpositions $(1\ 2),(2\ 3),\ldots,(n-1\ n)$, together 
with $(n\ \zeta n)$, where for example $(n\ \zeta n)$ is short 
for the $n$-by-$n$ matrix whose $i$-th column is $\zeta e_i$ if 
$i=n$ and $e_i$ otherwise. In general, a reflection is conjugate 
to a power of a generating reflection. This gives a total of $m$ 
conjugacy classes of reflections: one indexed by  $(n-1\ n)$ and 
the others by $(n\ \zeta^kn)$ for $1\leq k\leq m-1$. The Milnor 
fiber complex $\Delta$ of $G(m,1,n)$ is realized (see~\cite{O:Milnor})
as the union $\Delta=\cup gC$ of all translates $gC$ of the simplex 
\begin{equation}\label{Eq:Monomial:Chamber}
C=\{\alpha_1 b_1+\alpha_2 b_2+\ldots + \alpha_n b_n : 
\alpha_1+\ldots+\alpha_n=1,\ \text{$\alpha_i$ nonnegative real}\}
\end{equation}
where $b_k=(e_1+\ldots+e_k)/k$. This leads to the following 
convenient description of the Milnor fiber complex. 
The flag complex $\Delta(P)$ of a finite poset $P$ is the simplicial 
complex with  elements of $P$ for vertices and  
flags $\{x_1<x_2<\ldots<x_k : x_i\in P \}$ for simplices.

\begin{prop}\label{Proposition:Monomial:Comb}
Let $\Delta_n$ be the Milnor fiber complex of $G(m,1,n)$. 
Let $r$ be a reflection in $G(m,1,n)$. Let 
 $P_n$ be the collection of sets 
$\{\alpha_1e_{i_1},\ldots, \alpha_k e_{i_k}\}$
ordered by inclusion, 
where the $\alpha_i$'s are $m$-th roots of unity, $ i_1<\ldots<i_k$, and 
$1\leq k\leq n-1$. 
Then 
\begin{enumerate}[\rm(i)]
\item $\Delta_n$ is equivariantly isomorphic to the 
flag complex of $P_n$.\label{Monomial:Flag:Complex}
\item The subposet $P_n^r=\{X\in P_n : rX=X\}$ is isomorphic 
to $P_{n-1}$.\label{Monomial:Fixed:Poset}
\end{enumerate}
In particular, each wall of $\Delta_n$ is isomorphic 
to $\Delta_{n-1}$. 
\end{prop}

\begin{proof}
As an abstract simplicial complex 
 $\Delta_n$ is generated by the translates of the chamber $C=\{b_1,\ldots,b_n\}$. 
Consider the mapping that takes $b_k=(e_1+\ldots+e_k)/k$ to 
the set $\{e_1,\ldots,e_k\}$, so that a face $\{b_{i_1},\ldots, b_{i_k}\}$  of $C$
(where $i_1<i_2<\ldots<i_k$) goes to the flag 
\[\{e_1,\ldots, e_{i_1}\}\subset \{e_1,\ldots, e_{i_2}\}\subset
\ldots \subset \{e_1,\ldots,e_{i_k}\}.\]
Extend by the action of $G(m,1,n)$ to 
an isomorphism from $\Delta_n$ onto $\Delta(P_n)$. 
Hence~\eqref{Monomial:Flag:Complex}. 

The reflection $r$ is conjugate to either $t=(n-1\ n)$ or 
$s=(n\ \xi n)$ for some root of unity $\xi\neq 1$. 
Hence $P_n^r$ is isomorphic to either $P_n^s$ or $P_n^t$. 
We show that $P_n^s$ and $P_n^t$ are isomorphic to $P_{n-1}$. 
Consider $X\in P_n$ and write $X=\{\alpha_1e_{i_1},\ldots, \alpha_k e_{i_k}\}$. 
The reflection $t$ fixes $X$  
if and only if either $X$ contains both $\alpha e_{n-1}$ and $\alpha e_n$ 
for some $\alpha\in\CC$, or $X$ 
contains neither $\alpha e_{n-1}$ nor $\alpha e_n$ for any $\alpha\in\CC$. Hence 
the following is an isomorphism:
\[P_n^t\to P_{n-1}\text{ given by }X\mapsto X\setminus \CC e_n.\]
The set $X$ is fixed by $s$ if and only if $X=X\setminus \CC e_n$, so 
$P_n^s$ is isomorphic to $P_{n-1}$ as well. Hence~\eqref{Monomial:Fixed:Poset}.  
The last statement follows: 
$\Delta_n^r\cong\Delta(P_n^r)\cong \Delta(P_{n-1})\cong\Delta_{n-1}$. 
\end{proof}

\subsection{Proof of Theorem~\ref{Theorem:A}}\label{Section:Proof:A}
It suffices to assume that $G$ is irreducible by Proposition~\ref{Proposition:Reduction:A}.
\subsubsection*{Coxeter diagram with no subdiagram of type $D_4$, $F_4$, or $H_4$.} 
In this case, Abramenko tells us that each wall is a Coxeter complex,  
and hence a Milnor fiber complex.
\subsubsection*{Coxeter diagram with a subdiagram of type $D_4$, $F_4$, or $H_4$.} 
Since $\Delta$ is a Coxeter complex, Proposition~\ref{Lemma:Coxeter:Walls}
tells us that any wall of $\Delta$ that is a Milnor fiber complex must be 
a Coxeter complex. But Abramenko tells us that not all walls of $\Delta$ are 
Coxeter complexes.
Hence not all walls of $\Delta$ are Milnor fiber complexes. 
\subsubsection*{Full monomial groups $G(m,1,n)$ $(m\geq 2)$.} 
Proposition~\ref{Proposition:Monomial:Comb} says that 
any wall of the Milnor fiber complex of $G(m,1,n)$ is isomorphic to the Milnor fiber complex 
of ${G(m,1,n-1)}$. 
\subsubsection*{Groups of rank 1 and 2.} This case is clear.
\subsubsection*{The remaining exceptional groups: $G_{25},G_{26},G_{32}$.} 
Here we use Orlik's theorem (see \S\ref{SS:Top}), 
Proposition~\ref{Cor:Orlik}, and the cell counts 
of Theorem~\ref{Theorem:Chamber:Count}.

{\it Group $G_{25}$}. This is the rank-3 group with symbol
$3[3]3[3]3$. The basic degrees are $6,9,12$. 
Fix a wall $\Delta^r$. It is $1$-dimensional with $54$ chambers 
by~\eqref{Eq:Chamber:Count}, and it has   
the homotopy type of a $5^2$-fold bouquet of $1$-spheres by 
Orlik's theorem.  
If the wall is a Milnor fiber complex $\Delta(W,S)$, then 
because $\Delta(W,S)$ has dimension $|S|-1$ with chambers 
indexed by the elements of $W$, the group $W$ must be 
a rank-$2$ group with $54$ elements. If $W$ is 
reducible, then it  
must therefore be of the form $Z_j\times Z_k$ for some  
$j,k\in\NN$ such that $jk=54$ and 
$(j-1)(k-1)=5^2$ by Proposition~\ref{Cor:Orlik}. 
This is impossible. 
So $W$ must be irreducible with basic degrees $d_1,d_2$ satisfying 
 $d_1d_2=54$ and $(d_1-1)^2=5^2$. But from the classification 
(see Table~\ref{P:Table}) we find that there 
is no irreducible Coxeter or Shephard group of rank 2 
whose basic degrees are $6$ and $9$. So no wall of the Milnor fiber 
complex of $G_{25}$ is again a Milnor fiber complex.

{\it Group $G_{26}$}. 
This is the rank-3 group with symbol $3[3]3[4]2$. The basic degrees are $6,12,18$. 
Consider the wall $\Delta^{r_1}$ cut out by the generator $r_1$ that    
commutes with the one of order~$2$. It is $1$-dimensional with $72$ chambers and 
the homotopy type of a $5^2$-fold bouquet of $1$-spheres. 
If it is a Milnor fiber complex $\Delta(W,S)$, then $W$ is a rank-2 group 
with $72$ elements. $W$ must be irreducible because otherwise 
it is of the form $Z_j\times Z_k$ 
and there are no integers $j,k\in\NN$ such that $jk=72$ and $(j-1)(k-1)=5^2$. 
Thus $W$ is an irreducible rank-2 group whose basic degrees $d_1,d_2$ satisfy 
$d_1d_2=72$ and $(d_1-1)^2=5^2$. Hence $d_1=6$ and $d_2=12$. 
According to the classification (see Table~\ref{P:Table}) there 
are only two such groups: the group $G_5$ whose symbol is $3[4]3$, and 
the group $G(6,1,2)$ whose symbol is $2[4]6$. 

Vertices in $\Delta(W,S)$ are cosets $w\langle s_i\rangle$ in $W$ 
of the cyclic groups $\langle s_i\rangle$ for $s_i\in S$. Edges 
of $\Delta(W,S)$ are the cosets $\{w\}$ ($w\in W$) and 
the incidence relation is containment. So $\Delta(W,S)$ has  
 $|W|/|\langle s_i\rangle|$ vertices of degree $|\langle s_i\rangle|$ for 
$i=1,2$, and this accounts for all vertices.  
It follows from this discussion that 
the Milnor fiber complex of $G_5$ is 3-regular (all vertices have degree 3), and 
the Milnor fiber complex of $G(6,1,2)$ has $12$ vertices of degree $6$ 
and $36$ vertices of degree $2$. We claim that these vertex-degree 
distributions are different from 
the vertex-degree distribution in $\Delta^{r_1}$, and in turn 
$\Delta^{r_1}$ is not a Milnor fiber complex. 
To this end it is enough to show 
that there is a vertex of degree 4 in $\Delta^{r_1}$. 

Consider the vertex of $\Delta$ indexed by $H=\langle r_1,r_2\rangle$. 
This vertex is fixed (under left multiplication) by $r_1$ 
and therefore belongs to $\Delta^{r_1}$. We claim that it has 
degree 4 in $\Delta^{r_1}$. 
In $\Delta$ the edges incident to $H$ are the cosets of $\langle r_2\rangle$ 
and $\langle r_1\rangle$ in~$H$. The number of these cosets 
fixed by $r_1$ therefore equals the degree of $H$ as a vertex in $\Delta^{r_1}$. 
Since 
$H$ is the group $G_4$ with symbol $3[3]3$ whose 
basic degrees are $4$ and $6$, Eq.~\eqref{Eq:Chamber:Count} 
tells us that the number of these cosets in $H$ fixed by 
$r_1$ equals $4$. This concludes the present case.

{\it Group $G_{32}$}. 
This is the rank-4 group with symbol $3[3]3[3]3[3]3$ and basic degrees $12,18,24,30$. 
Fix a wall. It is $2$-dimensional with $5184$ chambers and 
the homotopy type of a $11^3$-fold bouquet of $2$-spheres. 
Suppose that the wall is a Milnor fiber complex $\Delta(W,S)$.  
Then $W$ must be a rank-3 group with $5184$ elements. 
It can not be a product of 3 rank-1 groups $Z_i\times Z_j\times Z_k$ 
because no $i,j,k\in\NN$ satisfy $ijk=5184$ 
and $(i-1)(j-1)(k-1)=11^3$. 
And it can not be the product of a rank-1 
group $Z_k$ and an irreducible rank-2 group $H$ because 
then for $d_1$ the smallest basic degree of $H$ we would have  
that $k|H|=5184$ and $(k-1)(d_1-1)^3=11^3$, so that $k=12$, $d_1=12$, and $|H|=432$, while 
the only irreducible rank-2 Coxeter or Shephard group with $432$ elements 
is a dihedral group whose smallest basic degree equals $2$. 
Hence $W$ must be an irreducible rank-3 Coxeter or Shephard group. 
Therefore the smallest basic degree $d_1$ of $W$ must satisfy 
$(d_1-1)^3=11^3$. But  Table~\ref{P:Table} shows that no irreducible rank-3 Coxeter or 
Shephard group has $5184$ elements and smallest basic degree $d_1$ equal 
to $12$.\qed

\section{Milnor walls}\label{S:Milnor:Walls}
\noindent
Write $\overline{\Omega}=\cup_{\sigma\in\Omega}\{\tau\in\Delta : \tau\subset \sigma\}$
for the simplicial complex 
generated by a family of simplices $\Omega\subset \Delta$. 
Let $\binom{R}{k}$ be the set of 
all $k$-element subsets of $R$, so that  
\[\binom{R}{k}=\{\type \sigma : \sigma\in \Delta,\ \dim \sigma=k-1\}.\]

\begin{Def}
A wall $\Delta^r$ of $\Delta$ is a {\it Milnor wall} 
if for some $\mathcal F\subset \binom{R}{n-1}$
the subcomplex 
\begin{equation}
(\Delta^r)_{\mathcal F} =\overline{\{\sigma\in \Delta^r : \type \sigma\in \mathcal F\}}
\end{equation}
is a Milnor fiber complex of dimension $\rk-2$. 
A {\it proper Milnor wall} is a Milnor wall 
that is not a Milnor fiber complex.
\end{Def}

Any wall $\Delta^r$ can be written 
as $(\Delta^r)_{\mathcal F}$ for $\mathcal F=\binom{R}{n-1}$. 
Therefore walls that are Milnor fiber complexes are Milnor walls. 
These are the non-proper Milnor walls. They 
are the only kind of Milnor wall found in Coxeter complexes. This 
is the following lemma.

\begin{lem}\label{Coxeter:Milnor:Walls}
Coxeter complexes have no proper Milnor walls.
\end{lem}

\begin{proof}
A wall in an $(n-1)$-dimensional Coxeter complexes is an $(n-2)$-sphere, 
and an $(n-2)$-dimensional Milnor fiber complex is a bouquet of $(n-2)$-spheres. 
Therefore, removing any chambers from a wall of an $(n-1)$-dimensional Coxeter complex gives 
something that is not a Milnor fiber complex of dimension $\rk-2$.
\end{proof}

\begin{prop}\label{Corollary:Milnor:Coxeter:Walls}
A wall of a Coxeter complex is a Milnor wall if and only if the wall is 
a Coxeter complex. 
\end{prop}

\begin{proof}
Suppose $\Sigma$ is a Milnor wall of a Coxeter complex. 
Then Lemma~\ref{Coxeter:Milnor:Walls} implies that 
$\Sigma$ is a Milnor fiber complex. 
Since $\Sigma$ is a wall of a Coxeter complex, it follows 
from Proposition~\ref{Lemma:Coxeter:Walls} that $\Sigma$ 
is a Coxeter complex.
The other direction is clear: 
a wall of a Coxeter complex that is itself a Coxeter complex is a 
(non-proper) Milnor wall. 
\end{proof}

\begin{prop}\label{Reduction:B}
Each wall of $\Delta(G,R)$ is a Milnor wall if and only if 
each wall of each $\Delta(G_i,R_i)$ is a Milnor wall. 
\end{prop}

\begin{proof}
Consider a wall 
$\Delta^r=\Delta(G_1,R_1) * \ldots *\Delta(G_i,R_i)^t * \ldots * \Delta(G_N,R_N)$, 
so that $t\in G_i$ is a reflection, and write $\rk_i=|R_i|$. 

Suppose 
$\Delta(G_i,R_i)^t$ is a Milnor wall, so that 
$(\Delta(G_i,R_i)^t)_{\mathcal F_i}$ is a Milnor fiber complex of dimension 
$\rk_i-2$ for some $\mathcal F_i\subset \binom{R_i}{n_i-1}$.
Put $\mathcal F=\{(R\setminus R_i) \cup S: S\in \mathcal F_i\}$, so that  
\begin{equation}\label{Eq:r:F}
(\Delta^r)_{\mathcal F}=
\Delta(G_1,R_1) * \ldots *(\Delta(G_i,R_i)^t)_{\mathcal F_i}* \ldots * \Delta(G_N,R_N).
\end{equation}
Since Theorem~\ref{Theorem:Definition}\eqref{Join:Isom}
tells us that a join of Milnor fiber complexes is again a Milnor fiber complex,
it follows that $(\Delta^r)_{\mathcal F}$
is a Milnor fiber complex of dimension $\rk-2$. Hence $\Delta^r$ is a Milnor wall. 

Now suppose $\Delta^r$ is a Milnor wall. Then 
$(\Delta^r)_{\mathcal F}$ is an $(\rk-2)$-dimensional Milnor fiber complex 
for some $\mathcal F\subset \binom{R}{n-1}$. 
The subcomplex $\Delta(G_i,R_i)^t$ has dimension $|R_i|-2$, and hence no simplex of type $R_i$. 
So the wall $\Delta^r$ has no simplex of type $R\setminus \{s\}$ for $s\in R_j$, $j\neq i$. 
Hence $\mathcal F\subset \{R\setminus\{s\}: s\in R_i\}$. 
But then  
for $\mathcal F_i=\{S\cap R_i: S\in \mathcal F\}$ we have \eqref{Eq:r:F}.
Since $(\Delta^r)_{\mathcal F}$ is an $(\rk-2)$-dimensional Milnor fiber complex 
in which $(\Delta(G_i,R_i)^t)_{\mathcal F_i}$ is the link of any simplex of 
type $R\setminus R_i$, it follows from Proposition~\ref{Lemma:Products} that 
$(\Delta(G_i,R_i)^t)_{\mathcal F_i}$ is a Milnor fiber complex of dimension $n_i-2$. 
So $\Delta(G_i,R_i)^t$ is a Milnor wall. 
\end{proof}

\subsection{Proof of Theorem~\ref{Theorem:B}} 
It suffices to assume that $G$ is irreducible by Proposition~\ref{Reduction:B}.

Suppose that the diagram of $G$ contains no subdiagram of type $D_4$, $F_4$, or $H_4$, 
and that $G$ is not $G_{25}$, $G_{26}$, or $G_{32}$. 
Then Theorem~\ref{Theorem:A} tells us that each wall 
of $\Delta$ is a Milnor fiber complex. Hence each wall of $\Delta$ is a Milnor wall.

Suppose now the diagram of $G$ contains a subdiagram of type $D_4$, $F_4$, or $H_4$.  
By the irreducibility of $G$ and the classification (Table~\ref{P:Table}),  
$G$ must be a Coxeter group. 
Then Abremenko's result tells us that 
not all walls of $\Delta$ are Coxeter complexes, 
and  
Proposition~\ref{Corollary:Milnor:Coxeter:Walls}
tells us that the Milnor walls of $\Delta$ are 
the walls of $\Delta$ that are again Coxeter complexes.
So in this case we conclude that not all walls of $\Delta$ are Milnor walls.

Finally suppose $G$ is $G_{25}$, $G_{26}$, or $G_{32}$. 
The object is to show that all walls are Milnor walls. 
The remainder of the proof explains how this claim follows 
from Orlik's original construction of the Milnor fiber complex $\Delta$ 
together with some observations made by Coxeter about the regular 
complex polytopes associated to $G_{25}$, $G_{26}$, and $G_{32}$.

A regular complex polytope~\cite[p.~115]{C} is a certain collection $\mathscr P$ 
of affine subspaces in $\CC^\rk$ with incidence 
relation given by proper inclusion subject to some conditions.  
The symmetry group of a regular complex polytope is a Shephard group $G$  
and the Milnor fiber complex of the Shephard group is 
constructed in \cite{O:Milnor} as a geometric realization of 
the flag complex of the poset of simplices of $\mathscr P$, 
whose $k$-simplices are the flags 
${\mathscr F}=(F^{(0)}\subsetneq F^{(1)}\subsetneq \ldots \subsetneq F^{(k)})$ 
of faces ${F^{(i)}\in \mathscr P}$; see \cite[Thm.~5.1]{O:Milnor}.
Index the generators of the Shephard group starting with $0$ instead of $1$, 
so that ${R=\{r_0,r_1,\ldots,r_{\rk-1}\}}$. Then 
a flag $\mathscr F$ corresponds in the Milnor fiber complex 
to a coset of $\langle R\setminus\{r_{\dim F}:F\in \mathscr F\}\rangle$, 
whose type is $\{r_{\dim F} : F\in \mathscr F\}$; see \cite{M0}.

Coxeter considered the {\it section 
of $\mathscr P$ by a reflecting hyperplane $H$}. This is the set 
\begin{equation}
\{F\in\mathscr P : F\subset F'\subset H\ \text{for some }F'\in\mathscr P\ \text{such that }
\dim F'=n-2\}.\end{equation}
He observed the following~\cite[pp. 123, 132]{C}:   
if $G=G_{25}$, then each section of $\mathscr P$ by a reflecting hyperplane 
is a regular complex polytope for $G(3,1,2)$; 
if $G=G_{26}$, then each section of $\mathscr P$ by a reflecting hyperplane is 
a regular complex polytope for either $G(3,1,2)$ or $G(6,1,2)$; 
and if $G=G_{32}$, then each section of $\mathscr P$ by a reflecting hyperplane is 
a regular complex polytope for $G_{26}$. 
These observations translate into the following statements about 
$\Sigma=(\Delta(G,R)^r)_{\mathcal F}$ 
for $\mathcal F=\{R\setminus\{r_{n-1}\}\}$ and $r\in G$ a reflection.
\begin{enumerate}[a.]
\item If $G=G_{25}$, then $\Sigma$ is isomorphic to the Milnor fiber complex 
of $G(3,1,2)$.
\item If $G=G_{26}$, then $\Sigma$ is isomorphic to 
the Milnor fiber complex of $G(3,1,2)$ or $G(6,1,2)$.
\item If $G=G_{32}$, then $\Sigma$ is isomorphic to the Milnor fiber 
complex of $G_{26}$. 
\end{enumerate}
Hence if $G$ is $G_{25}$, $G_{26}$, or $G_{32}$, then all walls of $\Delta(G,R)$ are Milnor walls. 
\qed

\begin{table}
\centering
\begin{tabular}{l@{\qquad\qquad}l@{\qquad\qquad}l}
\toprule
$G$ & Symbol/diagram & Basic degrees \\
\midrule
$Z_m$ & $m$& $m$ \\
$I_2(2m)$  & $2[2m]2$ &$2,2m$ \\
$I_2(2m-1)$ & $2[2m-1]2$&$2,2m-1$\\ 
$G_4$ & $3[3]3$& $4,6$ \\
$G_5$ & $3[4]3$& $6,12$ \\
$G_6$ & $3[6]2$& $4,12$ \\
$G_8$ & $4[3]4$& $8,12$\\
$G_9$ & $4[6]2$& $8,24$\\
$G_{10}$ & $4[4]3$& $12,24$\\
$G_{14}$ & $3[8]2$& $6,24$\\
$G_{16}$ & $5[3]5$& $20,30$\\
$G_{17}$ & $5[6]2$& $20,60$\\
$G_{18}$ & $5[4]3$& $30,60$\\
$G_{20}$ & $3[5]3$& $12,30$\\
$G_{21}$ & $3[10]2$& $12,60$\\
$G_{23}$ \text{\,\,\rm($H_3$)} & $2[3]2[5]2$& $2,6,10$\\ 
$G_{25}$ & $3[3]3[3]3$& $6,9,12$ \\
$G_{26}$ & $3[3]3[4]2$& $6,12,18$ \\
$G_{28}$ \text{\,\,\rm($F_4$)} & $2[3]2[4]2[3]2$& $2,6,8,12$ \\
$G_{30}$ \text{\,\,\rm($H_4$)} & $2[3]2[3]2[5]2$& $2,12,20,30$\\
$G_{32}$ & $3[3]3[3]3[3]3$& $12,18,24,30$ \\
$G_{35}$ \text{\,\,\rm($E_6$)} &
\raisebox{0pt}[20pt][5pt]{\begin{dynkin}
    \foreach \x in {1,...,5}
    {
        \dynkindot{\x}{0}
    }
    \dynkindot{3}{1}
    \dynkinline{1}{0}{5}{0}
    \dynkinline{3}{0}{3}{1}
  \end{dynkin}}
 & $2,5,6,8,9,12$ \\
$G_{36}$ \text{\,\,\rm($E_7$)}& 
\raisebox{0pt}[15pt][5pt]{\begin{dynkin}
    \foreach \x in {1,...,6}
    {
        \dynkindot{\x}{0}
    }
    \dynkindot{3}{1}
    \dynkinline{1}{0}{6}{0}
    \dynkinline{3}{0}{3}{1}
  \end{dynkin}}
& $2,6,8,10,12,14,18$ \\
$G_{37}$ \text{\,\,\rm($E_8$)}& 
\raisebox{0pt}[15pt][10pt]{\begin{dynkin}
    \foreach \x in {1,...,7}
    {
        \dynkindot{\x}{0}
    }
    \dynkindot{3}{1}
    \dynkinline{1}{0}{7}{0}
    \dynkinline{3}{0}{3}{1}
  \end{dynkin}}
&$2,8,12,14,18,20,24,30$ \\
$A_n$ & $2[3]2\ldots 2[3]2[3]2$ & $2,3,\ldots,n+1$ \\
$D_{n+1}$ & 
\raisebox{-4pt}[20pt][15pt]{\begin{dynkin}
    \foreach \x in {1,2,3,5,6}
    {
        \dynkindot{\x}{0}
    }
    \dynkindot{7}{-.5}
    \dynkindot{7}{.5}
    \dynkinline{1}{0}{2}{0}
    \dynkinline{2}{0}{3}{0}
    \fill (1.6cm,0) node {$\normalsize\ldots$};
    \dynkinline{5}{0}{6}{0}
    \dynkinline{6}{0}{7}{.5}
    \dynkinline{6}{0}{7}{-.5}
  \end{dynkin}}
& $2,4,\ldots,2n,n+1$\\
$G(m,1,n)$ & $2[3]2\ldots 2[3]2[4]m$ & $m,2m,\ldots,nm$ \\
\bottomrule
\end{tabular}
\caption{
($m\geq 2$) 
The finite irreducible Coxeter and Shephard groups, 
their diagrams, and their basic degrees.}
\label{P:Table}
\end{table}

\end{document}